\documentclass[12pt,a4paper]{amsart}
 \usepackage{amssymb,amsthm}
 \usepackage{multirow}
 \usepackage{dsfont}
 \newtheorem{thm}{Theorem}[section]

\newtheorem{prop}[thm]{Proposition}
\newtheorem{lem}[thm]{Lemma}

\newtheorem{rem}[thm]{Remark}

\theoremstyle{definition}
\newtheorem{defn}[thm]{Definition}

\begin{document}
 \title[Effects of edge addition or removal]
 {Effects of edge addition or removal on the nullity of a graph}
 \author{Ahmet Batal}
 \address{Department of Mathematics, Izmir Institute of Technology, 35430, Urla, Izmir, TURKEY}
\email{ahmetbatal@iyte.edu.tr}
 \date{}
\begin{abstract}
Lights Out is a game which can be played on any graph $G$. Initially we have a configuration which assigns one of the two states on or off to each vertex. The aim of the game is to turn all vertices to off state for an initial configuration by activating some vertices where each activation switches the state of the vertex and all of its neighbors. If the aim of the game can be accomplished for all initial configurations then $G$ is called always solvable. We call the dimension of the kernel of the closed neighborhood matrix of the graph over the field $\mathbb{Z}_2$, nullity of $G$. It turns out that $G$ is always solvable if and only if its nullity is zero. Moreover, the number of solutions of a given configuration is also determined by the nullity. We investigate the problem of how nullity changes when an edge is added to or removed from a graph. As a result we show that for every graph with positive nullity there exists an edge whose removal decreases the nullity. Conversely, we show that for every always solvable graph which is not an even graph with odd order, there exists an edge whose addition increases the nullity. We also show that if an always solvable graph is not even, then there is an edge whose removal increases the nullity.
\end{abstract}

 \maketitle
\section{Introduction}
Lights Out is a game which can be played on any undirected simple graph $G(V,E)$. The rules of the game are as follows. Initially every vertex has one of the two states \emph{on} or \emph{off}. Every push on a vertex switches the state of the vertex and all of its neighbors. This push is called \emph{the activation of the vertex}. The aim of the game is to turn all vertices to \emph{off} state by applying an activation pattern. It can be easily observed that the order in which the vertices are pushed is unimportant for the final result. Moreover pushing a vertex even number of times is equivalent to not pushing it and pushing a vertex odd number of times is equivalent to pushing it once. Therefore, any activation pattern can be identified by the set of activated vertices.  On the other hand, any initial configuration can be identified by the set of \emph{on} state vertices. Note also that there is a one to one correspondence between the subsets of $V(G)$ and elements of $\mathbb{Z}_2^n$ where $n$ is the order of $V(G)$. Indeed, if we enumerate $V(G)$ as $V(G)=\{v_1,...,v_n\}$ then every set $S$ can be characterized by its characteristic column vector $\mathbf{x}_S=(s_1,...,s_n)^t \in \mathbb{Z}_2^n,$ where $\mathbf{x}_S(v_i):=s_i=1$ if $v_i\in S$, and $s_i=0$ otherwise. Conversely, every element $\mathbf{x}$ of $\mathbb{Z}_2^n$ is the characteristic vector of its support set which consists of those vertices $v$ such that $\mathbf{x}(v)=1$. Hence, any activation pattern and initial configuration can be identified by the vectors of $\mathbb{Z}_2^n$ as well.

If an activation pattern set $P$ (activation pattern vector $\mathbf{p}$) turns all vertices to off state for a given initial configuration set $C$ (initial configuration vector $\mathbf{c}$) , then $C$ ($\mathbf{c}$) is called solvable and $P$ ($\mathbf{p}$) is called a solving pattern for $C$ ($\mathbf{c}$). We may also say $P$ solves $C$ ($\mathbf{p}$ solves $\mathbf{c}$). The graph $G$ is called \emph{always solvable} if every configuration is solvable.

For a vertex $v$, the set $N[v]=\{u\in V \;|\; u=v \;\text{or}\; u \;\text{is adjacent to}\; v\}$ is called \emph{the closed neighborhood set of vertex} $v$. Observe that $P$ solves $C$ if and only if $N[v] \cap P$ is odd for all $v\in C$ and even otherwise. We define the dot product between two vectors $\mathbf{x}$, $\mathbf{y} \in \mathbb{Z}_2^n$ as $\mathbf{x}\cdot \mathbf{y} := \mathbf{x}^t \mathbf{y}$. Then, the above observation can be rephrased as $P$ solves $C$ if and only if $\mathbf{x}_{N[v]} \cdot \mathbf{x}_P=1$ for all $v\in C$ and $0$ otherwise. As it is first observed by Sutner \cite{Sutner89},\cite{Sutner90}, this is equivalent to say that $P$ solves $C$ if and only if
\begin{equation}\label{Npa}
N\mathbf{x}_P=\mathbf{x}_C
\end{equation}
over the field $\mathbb{Z}_2$, where $N:=N(G)$ is the \emph{the closed neighborhood matrix of} $G$ whose $i$th column (equivalently row) is the characteristic vector of $N[v_i]$.

Let us denote the kernel of $N$  by $Ker(N)$. Elements of $Ker(N)$ are called \emph{null patterns}. $\nu(G):=dim(Ker(N(G))$ is called the nullity of $G$. The above formulation \eqref{Npa} between an initial configuration and its solving pattern leads us several observations \cite{Sutner89},\cite{Sutner90},\cite{Amin98}:

(\emph{O 1}) A graph $G$ is always solvable  \emph{iff} $\nu(G)=0 $.

(\emph{O 2}) Number of solving patterns for a given configuration is $2^{\nu(G)}$. Indeed, if $\mathbf{p}$ solves $\mathbf{c}$, then $\mathbf{p}+\boldsymbol{\ell}$ solves $\mathbf{c}$ as well for every null pattern $\boldsymbol{\ell}$.

(\emph{O 3}) A configuration vector (set) is solvable  \emph{iff} it (its characteristic vector) is orthogonal to every null pattern.

Suppose there is a vector $\mathbf{s}$ solving the all-ones configuration $\mathbf{1}:=(1,...,1)^t$. Then, the support set $S$ of $\mathbf{s}$ is an \emph{odd dominating set}, i.e.; $N[v] \cap S$ is odd for all $v\in V$. Therefore, we call $\mathbf{s}$ an \emph{odd dominating pattern}. Interestingly, every graph has an odd dominating pattern, or equivalently all-ones configuration $\mathbf{1}$ is solvable in every graph  \cite{Sutner89} (see also \cite{Caro96}, \cite{Cowen99}, \cite{Erikson04}).

How the nullity of a graph changes when a vertex is removed or what the nullity of the emergent graph becomes when two graphs are joined together by a single or multiple edges were investigated by several authors \cite{Amin02}, \cite{Ballard19}, \cite{Giffen13}, \cite{Edwards10} (see also \cite{Berman21} and \cite{Keough24} for recent articles on the subject). However, no previous studies appear to have investigated how nullity changes when an edge is added to or removed from a graph. In this paper we fill this gap. Indeed, we do this in a more general context. To understand the generality of our context let $G$ be a graph, $A_1$, $A_2$ be two disjoint subsets of $V(G)$. Then we define the following operation on $G$. For every pair of vertices $u\in A_1$ and $v\in A_2$ we add the edge $uv$ (the edge joining $u$ and $v$) to $G$ if $u$ and $v$ are not adjacent and we remove the edge $uv$ from $G$ if $u$ and $v$ are adjacent. So if $A_1$ and $A_2$ consist of single vertices, then the operation corresponds to simple edge removal or addition, depending on whether $u$ and $v$ are adjacent in $G$ or not. Then we investigate how nullity changes when this operation is applied to the graph under different scenarios as we stated in Proposition \ref{NN0} through Proposition \ref{HHN}. The results are summarized in Table \ref{table1}.

These results enable us to obtain important corollaries such as the existence of edges in a graph whose addition/removal inreases/decreases the nullity of the graph. More precisely, we show in Theorem \ref{thmr} that for a graph with nonzero nullity, there always exists an edge whose removal decreases the nullity. Conversely, we show in Theorem \ref{thma} that if an always solvable graph is not an even graph with odd order, then there exists an edge whose addition increases the nullity. Moreover, we prove in Theorem \ref{thmsss} that if an always solvable graph is not an even graph, then there exists an edge whose removal increases the nullity. In addition to these main theorems we prove several more results in one of which we give a characterization of always solvable graphs.

 \section{Nullity change under the general edge addition and removal operation}
\begin{lem}
\label{lem1}
Let $\mathbf{x}$ satisfy $\mathbf{x}\cdot \boldsymbol{\ell_0}=1$ for some null pattern $\boldsymbol{\ell_0}$. Then, $\mathbf{x}\cdot \boldsymbol{\ell}=1$ for exactly half of the null patterns $\boldsymbol{\ell}$.
\end{lem}
\begin{proof}
Equality of the number of null patterns $\boldsymbol{\ell}$ satisfying $\mathbf{x}\cdot \boldsymbol{\ell}=0$ and $\mathbf{x}\cdot \boldsymbol{\ell}=1$ follows from the fact that if a null pattern $\boldsymbol{\ell}$ satisfies $\mathbf{x}\cdot \boldsymbol{\ell}=0$, then the null pattern $\boldsymbol{\ell}+ \boldsymbol{\ell_0}$ satisfies $\mathbf{x}\cdot (\boldsymbol{\ell}+\boldsymbol{\ell_0})=1$; and if a null pattern $\boldsymbol{\ell}$ satisfies $\mathbf{x}\cdot \boldsymbol{\ell}=1$, then the null pattern $\boldsymbol{\ell}+ \boldsymbol{\ell_0}$ satisfies $\mathbf{x}\cdot (\boldsymbol{\ell}+\boldsymbol{\ell_0})=0$.
\end{proof}

Lemma \ref{lem1}, together with (\emph{O 2}) and (\emph{O 3}) gives us the following proposition.

\begin{prop}
\label{prop2}
Let $\mathbf{c}$ be a solvable configuration on a graph $G$ and $A$ be a subset of $V(G)$. Then, $A$ is not solvable if and only if $\mathbf{x}_A\cdot \mathbf{p}=1$ for exactly half of the solving patterns $\mathbf{p}$ for $\mathbf{c}$. $A$ is solvable if and and only if either

Case i) $\mathbf{x}_A\cdot \mathbf{p}=1$  for all solving patterns $\mathbf{p}$ for $\mathbf{c}$, or

Case ii) $\mathbf{x}_A\cdot \mathbf{p}=0$ for all solving patterns $\mathbf{p}$ for $\mathbf{c}.$
\end{prop}

Motivating ourselves by the above proposition we make the following definition.

\begin{defn}
\label{defn1}
We say a subset $A$ of $V(G)$ is half odd activated ($HO$) if it is not solvable. If $A$ is solvable, for a given solvable configuration $\mathbf{c}$, we say $A$ is $\mathbf{c}$-always odd activated ($\mathbf{c}$-$AO$), $\mathbf{c}$-never odd activated ($\mathbf{c}$-$NO$) if \emph{Case (i)} or \emph{Case (ii)} holds, respectively. In the case where $\mathbf{c}=\mathbf{1}$, instead of saying $A$ is $\mathbf{1}$-always odd activated, $\mathbf{1}$-never odd activated; we simply say $A$ is always odd activated ($AO$), never odd activated ($NO$), respectively. Further, we say a vertex $v$ is half activated, always activated, or never activated if $\{v\}$ is half odd activated, always odd activated, or never odd activated, respectively.
\end{defn}

\begin{rem}
\label{rem}
By (\emph{O 3}), we see that for a vertex $u$, $\{u\}$ is not solvable if and only if $\boldsymbol{\ell}(u)=\mathbf{x}_{\{u\}}\cdot \boldsymbol{\ell} =1$ for some null pattern $\boldsymbol{\ell}$.
\end{rem}

\begin{rem}
Let $\mathbf{s}$ be an odd dominating pattern and $\mathbf{p}$ be a solving pattern for $\mathbf{x}_A$. Then $\mathbf{x}_A\cdot \mathbf{s}= N\mathbf{p} \cdot \mathbf{s} = \mathbf{p} \cdot N\mathbf{s} =\mathbf{p} \cdot \mathbf{1}$. Hence, always odd activated sets are precisely those sets whose solving patterns have odd cardinality and never odd activated sets are precisely those sets whose solving patterns have even cardinality.
\end{rem}

For a given vector $\mathbf{x}$, let $\overline{\mathbf{x}}:=\mathbf{x}+\mathbf{1}$.

\begin{lem}
\label{lemx}
A vector $\mathbf{x}$ (a set $A$) is solvable if and only if $\overline{\mathbf{x}}$ ($A^c$) is solvable.
\end{lem}
\begin{proof}
Note that if a configuration $\mathbf{x}$ is solvable, then $\overline{\mathbf{x}}=\mathbf{1}+\mathbf{x}$ is solvable since it is the sum of two solvable configurations. Converse also holds since $\overline{\overline{\mathbf{x}}}=\mathbf{x}$. Note also that if $\mathbf{x}$ is the characteristic vector of a set $A$, then $\overline{\mathbf{x}}$ is the characteristic vector of $A^c$.
\end{proof}

\begin{lem}
\label{lem2}
For a given graph $G$ let the configuration $\mathbf{x}$ be solvable. Then $\mathbf{x}\cdot \mathbf{p}=0$ for all solving patterns $\mathbf{p}$ for $\overline{\mathbf{x}}$.
\end{lem}
\begin{proof}
Let $N$ be the closed neighborhood matrix of $G$, $\mathbf{p}$ be a solving patterns for $\overline{\mathbf{x}}$, and $\mathbf{s}$ be an odd dominating pattern of $G$. Then  $N(\mathbf{p}+\mathbf{s})=\overline{\mathbf{x}}+\mathbf{1}=\mathbf{x}$. Moreover $N$ can be written as $N=I+L+L^t$ where $I$ is the identity matrix and $L$ is a lower triangular matrix with $0$ diagonal entries. Consequently,

\begin{align}
  \mathbf{x}\cdot \mathbf{p} & = N(\mathbf{p}+\mathbf{s})\cdot \mathbf{p} \nonumber \\
                                        & = (N(\mathbf{p}+\mathbf{s}))^t \mathbf{p} \nonumber \\
                                        & = \mathbf{p}^t N \mathbf{p} +(N\mathbf{s})^t \mathbf{p} \nonumber \\
                                        & = \mathbf{p}^t \mathbf{p} + \mathbf{p}^t L\mathbf{p} +\mathbf{p}^t L^t \mathbf{p}+\mathbf{1}^t \mathbf{p} \\     \nonumber
                                        & = \mathbf{p} \cdot \mathbf{p} + \mathbf{p} \cdot (L\mathbf{p}) +(L\mathbf{p}) \cdot\mathbf{p} +\mathbf{1} \cdot  \mathbf{p} \\    \nonumber
                                        & = (\mathbf{p}+\mathbf{1})\cdot \mathbf{p} \\ \nonumber
                                        & = \overline{\mathbf{p}}\cdot \mathbf{p} \\ \nonumber
                                         & = 0.
\end{align}
\end{proof}

\begin{lem}
\label{prop1}
A set $A$ is $AO$ ($NO$) if and only if $A$ is $\mathbf{x}_A$-$AO$ ($\mathbf{x}_A$-$NO$).
\end{lem}
\begin{proof}
Let $\mathbf{s}$ be an odd dominating pattern and $\mathbf{r}$ be a solving pattern for $\mathbf{x}_A$. So $\mathbf{s}+\mathbf{r}$ solves $\mathbf{1}+\mathbf{x}_A=\overline{\mathbf{x}_A}$. By Lemma \ref{lem2}, $\mathbf{x}_A \cdot (\mathbf{s}+\mathbf{r})=0$. Hence, $\mathbf{x}_A \cdot \mathbf{s}= \mathbf{x}_A \cdot \mathbf{r}$.
\end{proof}

\begin{lem}
\label{lem3}
Let $G$ be a graph; $A_1$, $A_2$ be two disjoint solvable subsets of $V(G)$. Then the followings hold.

i) If $A_1$ and $A_2$ are both $AO$ or both $NO$, then $A_1$ is $\overline{\mathbf{x}_{A_2}}$-$NO$ if and only if $A_2$ is $\overline{\mathbf{x}_{A_1}}$-$NO$.

ii) If $A_1$ is $AO$ and $A_2$ is $NO$, then $A_1$ is $\overline{\mathbf{x}_{A_2}}$-$NO$ if and only if $A_2$ is $\overline{\mathbf{x}_{A_1}}$-$AO$.
\end{lem}
\begin{proof}
Let $\mathbf{p_1}$, $\mathbf{p_2}$ be solving patterns for $\overline{\mathbf{x}_{A_1}} $, $\overline{\mathbf{x}_{A_2}}$, respectively. Let us first assume that $A_1$ and $A_2$ are both $AO$ or both $NO$. Then $A_1\cup A_2$ is $NO$ since $A_1$ and $A_2$ are disjoint. By Lemma \ref{prop1} it is $\mathbf{x}_{A_1\cup A_2}$-$NO$. On the other hand, $N(\mathbf{p_1}+\mathbf{p_2})= \overline{\mathbf{x}_{A_1}}+\overline{\mathbf{x}_{A_2}}=\mathbf{x}_{A_1}+\mathbf{1}+\mathbf{x}_{A_1}+\mathbf{1}=\mathbf{x}_{A_1\cup A_2}$. Hence, $\mathbf{x}_{A_1\cup A_2}\cdot (\mathbf{p_1}+\mathbf{p_2})=0$. Explicitly  this means $\mathbf{x}_{A_1} \cdot \mathbf{p_1} + \mathbf{x}_{A_1} \cdot \mathbf{p_2} + \mathbf{x}_{A_2} \cdot \mathbf{p_1} + \mathbf{x}_{A_2} \cdot \mathbf{p_2}=0$. By Lemma \ref{lem2}, $\mathbf{x}_{A_1} \cdot \mathbf{p_1} = \mathbf{x}_{A_2} \cdot \mathbf{p_2} =0$. So  $\mathbf{x}_{A_1} \cdot \mathbf{p_2} + \mathbf{x}_{A_2} \cdot \mathbf{p_1}=0$, which implies $\mathbf{x}_{A_1} \cdot \mathbf{p_2} =\mathbf{x}_{A_2} \cdot \mathbf{p_1}$.

In the second case where $A_1$ is $AO$ and $A_2$ is $NO$, the proof follows the same lines as in the above case but this time  $A_1\cup A_2$ is $AO$ which leads us $\mathbf{x}_{A_1} \cdot \mathbf{p_2} + \mathbf{x}_{A_2} \cdot \mathbf{p_1}=1$. Hence, $\mathbf{x}_{A_2}\cdot \mathbf{p_1}=1-\mathbf{x}_{A_1}\cdot \mathbf{p_2}$.
\end{proof}

Let $G$ be a graph, $A_1$, $A_2$ be two disjoint subsets of $V(G)$. We define the graph $G^*$ as the resulting graph of the following operation. For every pair of vertices $u\in A_1$ and $v\in A_2$ we add the edge $uv$ to $G$ if $u$ and $v$ are not adjacent and we remove the edge $uv$ from $G$ if $u$ and $v$ are adjacent. Let $N$, $N^*$ be the closed neighborhood matrices of $G$, $G^*$, respectively.
Let $V(G)=\{v_1,...,v_n\}$. Realize that the relation between the closed neighborhood matrices $N$ and $N^*$ is given by $N^*=N+J$ where $J$ is the matrix with entries

\begin{align}
(J)_{ij}= \left\{ \begin{array}{cc}
                1 &  \;\;\text{if}\;\;\;\; v_i\in A_1,\; v_j\in A_2\; or \;v_j\in A_1,\; v_i\in A_2      \\
                0 &  \;\;\text{otherwise}  \\
                \end{array} \right\},
\end{align}
This implies that for any pattern $\mathbf{p}$  we have
\begin{equation}
\label{eqn1}
N^*\mathbf{p}=N\mathbf{p}+(\mathbf{x}_{A_1}\cdot \mathbf{p})\mathbf{x}_{A_2}+ (\mathbf{x}_{A_2}\cdot \mathbf{p})\mathbf{x}_{A_1}.
\end{equation}

 Let $\Delta\nu=\nu(G^*)-\nu(G)$. In the following propositions we state solving patterns for which configurations in $G$ correspond to an odd dominating pattern of $G^*$ and how $\Delta \nu$ and activation types of $A_1$ and $A_2$ change under different cases.

 In the proofs of the following propositions we reserve the notations $\mathbf{p_1}$ and $\mathbf{p_2}$ for arbitrary solving patterns for $\overline{\mathbf{x}_{A_1}} $, $\overline{\mathbf{x}_{A_2}}$, respectively; and we reserve the notations $\mathbf{r_1}$ and $\mathbf{r_2}$ for arbitrary solving patterns for $\mathbf{x}_{A_1} $, $\mathbf{x}_{A_2}$, respectively.


\begin{prop}
\label{NN0}
Let $A_1$ and $A_2$ be $NO$ sets in $G$. Also assume that $A_2$ is $\overline{\mathbf{x}_{A_1}}$-$NO$ in $G$. Then for any pattern $\mathbf{p}$, $N^*\mathbf{p}=\mathbf{1}\; \text{if and only if}\; N\mathbf{p}=\mathbf{1}.$ Moreover,
$A_1$ and $A_2$ are $NO$ in $G^*$ and $\Delta \nu =0$.
\end{prop}
\begin{proof}
Let $N^*\mathbf{p}=\mathbf{1}$. First, assume that $\mathbf{x}_{A_1}\cdot \mathbf{p}=0 \;\text{and}\; \mathbf{x}_{A_2}\cdot \mathbf{p}=1$.
Then, by \eqref{eqn1}, $N\mathbf{p}=\overline{\mathbf{x}_{A_1}}.$ By the assumption of the proposition $\mathbf{x}_{A_2}\cdot \mathbf{p}=0$, which is a contradiction.

Second, assume that $\mathbf{x}_{A_1}\cdot \mathbf{p}=1 \;\text{and}\; \mathbf{x}_{A_2}\cdot \mathbf{p}=0.$
Then, by \eqref{eqn1}, $N\mathbf{p}=\overline{\mathbf{x}_{A_2}}.$ By Lemma \ref{lem3}, $A_2$ is $\overline{\mathbf{x}_{A_1}}$-$NO$ is equivalent to say $A_1$ is $\overline{\mathbf{x}_{A_2}}$-$NO$. Hence $\mathbf{x}_{A_1}\cdot \mathbf{p}=0$, which is a contradiction.

Third assume that $\mathbf{x}_{A_1}\cdot \mathbf{p}=1 \;\text{and}\; \mathbf{x}_{A_2}\cdot \mathbf{p}=1.$
Then, by \eqref{eqn1}, $N\mathbf{p}=\overline{\mathbf{x}_{A_1\cup A_2}}=\overline{\mathbf{x}_{A_1}} + \mathbf{x}_{A_2}$ , which means $\mathbf{p}=\mathbf{p_1}+\mathbf{r_2}$.
\begin{equation}
\label{eqn14}
1=\mathbf{x}_{A_2}\cdot \mathbf{p}=\mathbf{x}_{A_2}\cdot \mathbf{p_1}+\mathbf{x}_{A_2}\cdot \mathbf{r_2}=\mathbf{x}_{A_2}\cdot \mathbf{r_2}
\end{equation}
by the assumption.
On the other hand, $A_2$ is $NO$ implies  $A_2$ is $\mathbf{x}_{A_2}$-$NO$ by Lemma \ref{prop1}. Hence $\mathbf{x}_{A_2}\cdot \mathbf{r_2}=0$, which contradicts with \eqref{eqn14}.

Consequently, the only possibility is that $\mathbf{x}_{A_1}\cdot \mathbf{p}=0 \;\text{and}\; \mathbf{x}_{A_2}\cdot \mathbf{p}=0,$
which, by \eqref{eqn1}, gives $N\mathbf{p}=\mathbf{1}$.

Conversely, let $N\mathbf{p}=\mathbf{1}$. Then $\mathbf{p}$ is an odd dominating pattern of $G$. Since $A_1$ and $A_2$ are $NO$ in $G$, this gives $\mathbf{x}_{A_1}\cdot \mathbf{p}=0 \;\text{and}\; \mathbf{x}_{A_2}\cdot \mathbf{p}=0$, which together with \eqref{eqn1}, implies $N^*\mathbf{p}=\mathbf{1}$.

Since the odd dominating patterns of $G$ and $G^*$ are identical, the activation types of sets does not change. Hence $A_1$ and $A_2$ are NO in $G^*$. Also the fact that the  number of odd dominating patterns of $G$ and $G^*$ are the same, implies $\nu(G)=\nu(G^*)$ by (\emph{O 2}). Hence $\Delta \nu =0$.

\end{proof}
\begin{prop}
\label{NN1}
Let $A_1$ and $A_2$ be $NO$ sets in $G$. Also assume that $A_2$ is $\overline{\mathbf{x}_{A_1}}$-$AO$ in $G$. Then for any pattern $\mathbf{p}$, $N^*\mathbf{p}=\mathbf{1}\; \text{if and only if}\; N\mathbf{p}=\mathbf{1} \; \text{or}\;  N\mathbf{p}=\overline{\mathbf{x}_{A_1}} \; \text{or}\;  N\mathbf{p}=\overline{\mathbf{x}_{A_2}} \; \text{or}\;  N\mathbf{p}=\overline{\mathbf{x}_{A_1\cup A_2}}  .$ Moreover, $A_1$ and $A_2$ are $HO$ in $G^*$, and $\Delta \nu =2$.
\end{prop}
\begin{proof}
Let $N^*\mathbf{p}=\mathbf{1}$. Then, by \eqref{eqn1}, we have
$$N\mathbf{p}=\mathbf{1}+(\mathbf{x}_{A_1}\cdot \mathbf{p})\mathbf{x}_{A_2}+ (\mathbf{x}_{A_2}\cdot \mathbf{p})\mathbf{x}_{A_1}.$$
Hence, depending on the the value of $\mathbf{x}_{A_1}\cdot \mathbf{p}$ and $\mathbf{x}_{A_2}\cdot \mathbf{p}$, one of the four cases $N\mathbf{p}=1,\;  N\mathbf{p}=\overline{\mathbf{x}_{A_1}},\;  N\mathbf{p}=\overline{\mathbf{x}_{A_2}}, \; \text{or}\;  N\mathbf{p}=\overline{\mathbf{x}_{A_1\cup A_2}}$ must hold.

Conversely, let us first assume that $N\mathbf{p}=\mathbf{1}$. Then by assumption $\mathbf{x}_{A_1}\cdot \mathbf{p}=\mathbf{x}_{A_2}\cdot \mathbf{p}=0$. By \eqref{eqn1}, $N^*\mathbf{p}=\mathbf{1}$.

Second, let $N \mathbf{p}= \overline{\mathbf{x}_{A_1}}.$ Then by assumption, $\mathbf{x}_{A_2}\cdot \mathbf{p}=1$. Moreover, by Lemma \ref{lem2}, $\mathbf{x}_{A_1}\cdot \mathbf{p}=0$. Thus, by \eqref{eqn1}, $N^* \mathbf{p}= \overline{\mathbf{x}_{A_1}}+\mathbf{x}_{A_1}=\mathbf{1}$.

Third, let $N \mathbf{p}= \overline{\mathbf{x}_{A_2}}.$ By assumption and Lemma \ref{lem3}, $\mathbf{x}_{A_1}\cdot \mathbf{p}=1$. Moreover, by Lemma \ref{lem2}, $\mathbf{x}_{A_2}\cdot \mathbf{p}=0$. Then, by \eqref{eqn1}, $N^* \mathbf{p}= \overline{\mathbf{x}_{A_2}}+\mathbf{x}_{A_2}=\mathbf{1}$.

Lastly, let $N\mathbf{p}=\overline{\mathbf{x}_{A_1\cup A_2}}= \overline{\mathbf{x}_{A_1}}+\mathbf{x}_{A_2}=\mathbf{x}_{A_1}+\overline{\mathbf{x}_{A_2}}.$ Hence $\mathbf{p}=\mathbf{p_1}+\mathbf{r_2}=\mathbf{p_2}+\mathbf{r_1}$. Note that by assumption $A_1$ and $A_2$ are $NO$. By Lemma \ref{prop1}, $A_1$ and $A_2$ are  $\mathbf{x}_{A_1}$-$NO$ and $\mathbf{x}_{A_2}$-$NO$, respectively. Hence, $\mathbf{x}_{A_1}\cdot \mathbf{r_1}=\mathbf{x}_{A_2}\cdot \mathbf{r_2}=0$. Moreover, by assumption and Lemma \ref{lem3}, $\mathbf{x}_{A_2}\cdot \mathbf{p_1}=\mathbf{x}_{A_1}\cdot \mathbf{p_2}=1$. Then, $\mathbf{x}_{A_1}\cdot \mathbf{p} =\mathbf{x}_{A_1}\cdot \mathbf{p_2}+ \mathbf{x}_{A_1}\cdot \mathbf{r_1}=1$, and $\mathbf{x}_{A_2}\cdot \mathbf{p} =\mathbf{x}_{A_2}\cdot \mathbf{p_1}+ \mathbf{x}_{A_2}\cdot \mathbf{r_2}=1$. Consequently, by \eqref{eqn1}, $N^* \mathbf{p}= \overline{\mathbf{x}_{A_1\cup A_2}}+\overline{\mathbf{x}_{A_1}}+\overline{\mathbf{x}_{ A_2}}=\mathbf{1}.$

Note that there exist $2^{\nu(G)}$ patterns $\mathbf{p}$ for  each case $N\mathbf{p}=\mathbf{1},\;  N\mathbf{p}=\overline{\mathbf{x}_{A_1}},\;  N\mathbf{p}=\overline{\mathbf{x}_{A_2}},\; \text{and}\;  N\mathbf{p}=\overline{\mathbf{x}_{A_1\cup A_2}}$. So there exist total of $2^{\nu(G)+2}$ patterns satisfying one of the four cases. On the other hand there exist $2^{\nu(G^*)}$ odd dominating patterns of $G^*$. Hence, $\nu(G^*)=\nu(G)+2$, which implies $\Delta \nu =2$. Moreover, $\mathbf{x}_{A_1}\cdot \mathbf{p}$ is zero in the first two cases and one in the last two cases, which makes $A_1$ a $HO$ set in $G^*$. Similarly, $\mathbf{x}_{A_2}\cdot \mathbf{p}$ is zero in the first and third cases and one in the second and fourth cases. Hence, $A_2$ is $HO$ in $G^*$ as well.

\end{proof}

\begin{prop}
\label{NA0}
Let $A_1$ be a $NO$ and $A_2$ be a $AO$ set in $G$. Also assume that $A_2$ is $\overline{\mathbf{x}_{A_1}}$-$NO$ in $G$.
Then for any pattern $\mathbf{p}$, $N^*\mathbf{p}=\mathbf{1}\; \text{if and only if}\;  N\mathbf{p}=\overline{\mathbf{x}_{A_2}} \; \text{or}\;  N\mathbf{p}=\overline{\mathbf{x}_{A_1\cup A_2}}  .$ Moreover, $A_1$ is a $AO$ and $A_2$ is a $HO$ set in $G^*$, and $\Delta \nu =1$.
\end{prop}
\begin{proof}
Let $N^*\mathbf{p}=\mathbf{1}$. Assume first that $\mathbf{x}_{A_1}\cdot \mathbf{p}=\mathbf{x}_{A_2}\cdot \mathbf{p}=0$. Then, by \eqref{eqn1}, we have $N \mathbf{p}= \mathbf{1}$. But then $\mathbf{x}_{A_2}\cdot \mathbf{p}=1$ since $A_2$ is an $AO$ set in $G$, which is a contradiction.

Assume that $\mathbf{x}_{A_1}\cdot \mathbf{p}=0$, $\mathbf{x}_{A_2}\cdot \mathbf{p}=1$. By \eqref{eqn1}, $N \mathbf{p}= \overline{\mathbf{x}_{A_1}}$. Then $\mathbf{x}_{A_2}\cdot \mathbf{p}=0$ by the assumption of the proposition, which is a contradiction. Hence $N\mathbf{p}=\overline{\mathbf{x}_{A_2}} \; \text{or}\;  N\mathbf{p}=\overline{\mathbf{x}_{A_1\cup A_2}}  .$

Conversely, assume $N\mathbf{p}=\overline{\mathbf{x}_{A_2}}$. Then $\mathbf{x}_{A_1}\cdot \mathbf{p}=1$ by the assumption of the proposition and Lemma \ref{lem3}. Moreover, $\mathbf{x}_{A_2}\cdot \mathbf{p}=0$ by Lemma \ref{lem2}. Hence, by \eqref{eqn1}, $N^*\mathbf{p}=\mathbf{1}$.

Assume  $N\mathbf{p}=\overline{\mathbf{x}_{A_1\cup A_2}} = \overline{\mathbf{x}_{A_1}}+\mathbf{x}_{A_2}=\overline{\mathbf{x}_{A_2}}+\mathbf{x}_{A_1}$. Then $\mathbf{p}=\mathbf{p_1}+\mathbf{r_2}= \mathbf{p_2}+\mathbf{r_1}$. Moreover, $\mathbf{x}_{A_1}\cdot \mathbf{p_2}=1$ by the assumption of the proposition and Lemma \ref{lem3}. In view of Lemma \ref{prop1}, $\mathbf{x}_{A_1}\cdot \mathbf{r_1}=0$ since $A_1$ is a $NO$ set in $G$, and $\mathbf{x}_{A_2}\cdot \mathbf{r_2}=1$ since $A_2$ is a $AO$ set in $G$. Therefore, $\mathbf{x}_{A_1}\cdot \mathbf{p}=\mathbf{x}_{A_1}\cdot \mathbf{p_2}+\mathbf{x}_{A_1}\cdot \mathbf{r_1}=1$, and $\mathbf{x}_{A_2}\cdot \mathbf{p}=\mathbf{x}_{A_2}\cdot \mathbf{p_1}+\mathbf{x}_{A_2}\cdot \mathbf{r_2}=1$. Then,  by \eqref{eqn1}, $N^*\mathbf{p}=\mathbf{1}$.

There exist $2 .2^{\nu(G)}$ patterns satisfying $N\mathbf{p}=\overline{\mathbf{x}_{A_2}} \; \text{or}\;  N\mathbf{p}=\overline{\mathbf{x}_{A_1\cup A_2}},$ and there are $2^{\nu(G^*)}$ odd dominating patterns of $G^*$. Hence $\nu(G^*)=\nu(G)+1$, which gives $\Delta\nu=1$. $\mathbf{x}_{A_1}\cdot \mathbf{p}=1$ in both cases $N\mathbf{p}=\overline{\mathbf{x}_{A_2}} \; \text{and}\;  N\mathbf{p}=\overline{\mathbf{x}_{A_1\cup A_2}},$ which makes
$A_1$ an $AO$ set in $G^*$. On the other hand, $\mathbf{x}_{A_2}\cdot \mathbf{p}$ is $0$ in the case $N\mathbf{p}=\overline{\mathbf{x}_{A_2}}$ and $1$ in the case $N\mathbf{p}=\overline{\mathbf{x}_{A_1\cup A_2}}$. Hence
$A_2$ is a $HO$ set in $G^*$.
\end{proof}

\begin{prop}
\label{NA1}
Let $A_1$ be a $NO$ and $A_2$ be an $AO$ set in $G$. Also assume that $A_2$ is $\overline{\mathbf{x}_{A_1}}$-$AO$ in $G$.
Then for any pattern $\mathbf{p}$, $N^*\mathbf{p}=\mathbf{1}\; \text{if and only if}\;  N\mathbf{p}=\overline{\mathbf{x}_{A_1}}.$ Moreover, $A_1$ is a $NO$ and $A_2$ is an $AO$ set in $G^*$, and $\Delta \nu =0$.
\end{prop}
\begin{proof}
Let $N^*\mathbf{p}=\mathbf{1}$. First assume that $\mathbf{x}_{A_1}\cdot \mathbf{p}=0$ and  $\mathbf{x}_{A_2}\cdot \mathbf{p}=0$. Then, by \eqref{eqn1}, $N\mathbf{p}=\mathbf{1}$. This implies $\mathbf{x}_{A_2}\cdot \mathbf{p}=1$ by assumption, which is a contradiction.
\\
Second assume that $\mathbf{x}_{A_1}\cdot \mathbf{p}=1$ and  $\mathbf{x}_{A_2}\cdot \mathbf{p}=0$. Then, by \eqref{eqn1}, $N\mathbf{p}=\overline{\mathbf{x}_{A_2}}$. By assumption and Lemma \ref{lem3}, $\mathbf{x}_{A_1}\cdot \mathbf{p}=0,$ which is a contradiction.
\\
Third assume that  $\mathbf{x}_{A_1}\cdot \mathbf{p}=1$ and  $\mathbf{x}_{A_2}\cdot \mathbf{p}=1$. Then, by \eqref{eqn1}, $N\mathbf{p}=\overline{\mathbf{x}_{A_1}}+ \mathbf{x}_{A_2}$, which implies $\mathbf{p}=\mathbf{p_1}+\mathbf{r_2}$. Because $A_2$ is an $AO$ set in $G$, $\mathbf{x}_{A_2}\cdot \mathbf{r_2}=1$ by Lemma \ref{prop1}. Thus, $\mathbf{x}_{A_2}\cdot \mathbf{p}= \mathbf{x}_{A_2}\cdot \mathbf{p_1}+ \mathbf{x}_{A_2}\cdot \mathbf{r_2}=0$, which is a contradiction.
\\
So the only possible case is  $\mathbf{x}_{A_1}\cdot \mathbf{p}=0$ and $\mathbf{x}_{A_2}\cdot \mathbf{p}=1,$ which by \eqref{eqn1}, gives $N\mathbf{p}=\overline{\mathbf{x}_{A_1}}$.

Conversely, if $N\mathbf{p}= \overline{\mathbf{x}_{A_1}}$, then $\mathbf{x}_{A_2}\cdot \mathbf{p}=1$ by assumption. Moreover, $\mathbf{x}_{A_1}\cdot \mathbf{p}=0$ by Lemma \ref{lem2}. Thus, by \eqref{eqn1}, $N^*\mathbf{p}=\overline{\mathbf{x}_{A_1}}+\mathbf{x}_{A_1}=\mathbf{1}$.

There exist $2^{\nu(G)}$ patterns satisfying $N\mathbf{p}=\overline{\mathbf{x}_{A_1}}$ and there are $2^{\nu(G^*)}$ odd dominating patterns of $G^*$. Hence $\nu(G^*)=\nu(G)$, which gives $\Delta\nu=0$. Moreover $\mathbf{x}_{A_1}\cdot \mathbf{p}=0$ and $\mathbf{x}_{A_2}\cdot \mathbf{p}=1$  in case of $N\mathbf{p}=\overline{\mathbf{x}_{A_1}}$. Hence $A_1$ is a $NO$ and $A_2$ is an $AO$ set in $G^*$.
\end{proof}

\begin{prop}
\label{AA0}
Let $A_1$ and $A_2$ be two $AO$ sets in $G$. Also assume that $A_2$ is $\overline{\mathbf{x}_{A_1}}$-$NO$ in $G$. Then for any pattern $\mathbf{p}$, $N^*\mathbf{p}=\mathbf{1}\; \text{if and only if}\;  N\mathbf{p}=\overline{\mathbf{x}_{A_1\cup A_2}}.$ Moreover, $A_1$ and $A_2$ are $AO$ sets in $G^*$ , and $\Delta \nu =0$.
\end{prop}
\begin{proof}
Let $N^*\mathbf{p}=\mathbf{1}$. First assume that $\mathbf{x}_{A_1}\cdot \mathbf{p}=\mathbf{x}_{A_2}\cdot \mathbf{p}=0.$ Then, by \eqref{eqn1}, $N\mathbf{p}=\mathbf{1}$. Since $A_1$ is an $AO$ set in $G$ we must have $\mathbf{x}_{A_1}\cdot \mathbf{p}=1$, which is a contradiction.

Second assume that $\mathbf{x}_{A_1}\cdot \mathbf{p}=0$,  $\mathbf{x}_{A_2}\cdot \mathbf{p}=1.$ Then, by \eqref{eqn1}, $N\mathbf{p}=\overline{\mathbf{x}_{A_1}}$. So, by the assumption of the proposition $\mathbf{x}_{A_2}\cdot \mathbf{p}=0,$ which is a contradiction.

Third assume that $\mathbf{x}_{A_1}\cdot \mathbf{p}=1$,  $\mathbf{x}_{A_2}\cdot \mathbf{p}=0.$ Then,  by \eqref{eqn1}, $N\mathbf{p}=\overline{\mathbf{x}_{A_2}}$. By the assumption of the proposition and Lemma \ref{lem3}, $\mathbf{x}_{A_1}\cdot \mathbf{p}=0$, which is a contradiction.

So the last case where $\mathbf{x}_{A_1}\cdot \mathbf{p}=1$,  $\mathbf{x}_{A_2}\cdot \mathbf{p}=1$ is the only possible case, which leads to $N\mathbf{p}= \overline{\mathbf{x}_{A_1\cup A_2}}$ by \eqref{eqn1}.

Conversely, let $N\mathbf{p}= \overline{\mathbf{x}_{A_1\cup A_2}}=\overline{\mathbf{x}_{A_1}}+ \mathbf{x}_{A_2}= \overline{\mathbf{x}_{A_2}}+ \mathbf{x}_{A_1}.$ Then, $\mathbf{p}=\mathbf{p_1}+\mathbf{r_2}=\mathbf{p_2}+\mathbf{r_1}$. By assumption $A_1$ and $A_2$ are $AO$. By Lemma \ref{prop1}, $A_1$ and $A_2$ are  $\mathbf{x}_{A_1}$-$AO$ and $\mathbf{x}_{A_2}$-$AO$, respectively. Hence, $\mathbf{x}_{A_1}\cdot \mathbf{r_1}=\mathbf{x}_{A_2}\cdot \mathbf{r_2}=1$. Moreover, by assumption and Lemma \ref{lem3}, $\mathbf{x}_{A_2}\cdot \mathbf{p_1}=\mathbf{x}_{A_1}\cdot \mathbf{p_2}=0$. Then, $\mathbf{x}_{A_1}\cdot \mathbf{p} =\mathbf{x}_{A_1}\cdot \mathbf{p_2}+ \mathbf{x}_{A_1}\cdot \mathbf{r_1}=1$, and $\mathbf{x}_{A_2}\cdot \mathbf{p} =\mathbf{x}_{A_2}\cdot \mathbf{p_1}+ \mathbf{x}_{A_2}\cdot \mathbf{r_2}=1$. Consequently, by \eqref{eqn1}, $N^* \mathbf{p}= \overline{\mathbf{x}_{A_1\cup A_2}}+\overline{\mathbf{x}_{A_1}}+\overline{\mathbf{x}_{ A_2}}=\mathbf{1}.$

There exist $2^{\nu(G)}$ patterns satisfying $N\mathbf{p}=\overline{\mathbf{x}_{A_1\cup A_2}}$ and there are $2^{\nu(G^*)}$ odd dominating patterns of $G^*$. Hence, $\nu(G^*)=\nu(G)$, which gives $\Delta\nu=0$. Moreover $\mathbf{x}_{A_1}\cdot \mathbf{p}=1$ and $\mathbf{x}_{A_2}\cdot \mathbf{p}=1$  in case of $\overline{\mathbf{x}_{A_1\cup A_2}}$. Therefore,  $A_1$ and $A_2$ are $AO$ sets in $G^*$.
\end{proof}

\begin{prop}
\label{AA1}
Let $A_1$ and $A_2$ be two  $AO$ sets in $G$. Also assume that $A_2$ is $\overline{\mathbf{x}_{A_1}}$-$AO$ in $G$. Then for any pattern $\mathbf{p}$, $N^*\mathbf{p}=\mathbf{1}\; \text{if and only if}\;  N\mathbf{p}=\overline{\mathbf{x}_{A_1}} \; \text{or}\; N\mathbf{p}=\overline{\mathbf{x}_{A_2}}. $ Moreover, $A_1$ and $A_2$ are  $HO$ sets in $G^*$, and $\Delta \nu =1$.
\end{prop}
\begin{proof}
Let $N^*\mathbf{p}=\mathbf{1}$. First assume that $\mathbf{x}_{A_1}\cdot \mathbf{p}=\mathbf{x}_{A_2}\cdot \mathbf{p}=0.$ Then, by \eqref{eqn1}, $N\mathbf{p}=\mathbf{1}$. Since $A_1$ is an $AO$ set in $G$, we must have $\mathbf{x}_{A_1}\cdot \mathbf{p}=1$, which is a contradiction.

Second assume that $\mathbf{x}_{A_1}\cdot \mathbf{p}=\mathbf{x}_{A_2}\cdot \mathbf{p}=1.$ Then, by \eqref{eqn1}, $N\mathbf{p}=\overline{\mathbf{x}_{A_1\cup A_2}}= \overline{\mathbf{x}_{A_2}}+ \mathbf{x}_{A_1}$. Thus, $\mathbf{p}=\mathbf{p_2}+\mathbf{r_1}.$ By assumption of the proposition and Lemma \ref{lem3}, $\mathbf{x}_{A_1}\cdot \mathbf{p_2}=1$. Lemma \ref{prop1} and the fact that $A_1$ is an $AO$ set in $G$ imply $\mathbf{x}_{A_1}\cdot \mathbf{r_1}=1.$ Thus, $\mathbf{x}_{A_1}\cdot \mathbf{p}=\mathbf{x}_{A_1}\cdot \mathbf{p_2}+\mathbf{x}_{A_1}\cdot \mathbf{r_1}=1+1=0$, which is a contradiction.

So either the case where $\mathbf{x}_{A_1}\cdot \mathbf{p}=0$,  $\mathbf{x}_{A_2}\cdot \mathbf{p}=1$ or the case where $\mathbf{x}_{A_1}\cdot \mathbf{p}=1$,  $\mathbf{x}_{A_2}\cdot \mathbf{p}=0$
holds true, which, by \eqref{eqn1}, leads to $N\mathbf{p}= \overline{\mathbf{x}_{A_1}}$ or $N\mathbf{p}= \overline{\mathbf{x}_{A_2}}$, respectively.

Conversely, assume first that $N\mathbf{p}= \overline{\mathbf{x}_{A_1}}$. By the assumption of the proposition, $\mathbf{x}_{A_2}\cdot \mathbf{p}=1.$ By Lemma \ref{lem2},  $\mathbf{x}_{A_1}\cdot \mathbf{p}=0.$ Hence, by \eqref{eqn1}, $N\mathbf{p}= \overline{\mathbf{x}_{A_2}}+\mathbf{x}_{A_2}=\mathbf{1}$.

Second assume that $N\mathbf{p}= \overline{\mathbf{x}_{A_2}}$. By the assumption of the proposition and by Lemma \ref{lem3},  $\mathbf{x}_{A_1}\cdot \mathbf{p}=1.$ By Lemma \ref{lem2},  $\mathbf{x}_{A_2}\cdot \mathbf{p}=0.$ Hence, by \eqref{eqn1}, $N\mathbf{p}= \overline{\mathbf{x}_{A_2}}+\mathbf{x}_{A_2}=\mathbf{1}$.

There exist $2 .2^{\nu(G)}$ patterns satisfying $N\mathbf{p}=\overline{\mathbf{x}_{A_1}} \; \text{or}\;  N\mathbf{p}=\overline{\mathbf{x}_{A_2}} ,$ and there are $2^{\nu(G^*)}$ odd dominating patterns of $G^*$. Hence $\nu(G^*)=\nu(G)+1$, which gives $\Delta\nu=1$. $\mathbf{x}_{A_1}\cdot \mathbf{p}$ is $0$ in the case $N\mathbf{p}=\overline{\mathbf{x}_{A_1}}$ and $1$ in the case $N\mathbf{p}=\overline{\mathbf{x}_{A_2}}$, which makes $A_1$ a $HO$ set in $G^*$. On the other hand, $\mathbf{x}_{A_2}\cdot \mathbf{p}$ is $1$ in the case $N\mathbf{p}=\overline{\mathbf{x}_{A_1}}$ and $0$ in the case $N\mathbf{p}=\overline{\mathbf{x}_{A_2}}$. Hence, $A_2$ a $HO$ set in $G^*$ as well.

\end{proof}

\begin{prop}
\label{NH}
Let $A_1$ be a $NO$ and $A_2$ be a $HO$ set in $G$. Then for any pattern $\mathbf{p}$, $N^*\mathbf{p}=\mathbf{1}\; \text{if and only if}\;  N\mathbf{p}=\mathbf{1} \; \text{with}\; \mathbf{x}_{A_2}\cdot \mathbf{p}=0\; \text{or}\; N\mathbf{p}=\overline{\mathbf{x}_{A_1}}\;\text{with} \; \mathbf{x}_{A_2}\cdot \mathbf{p}=1.$ Moreover, $A_1$ is a $NO$ and $A_2$ is a $HO$ set in $G^*$, and $\Delta \nu =0$.
\end{prop}
\begin{proof}
Let $N^*\mathbf{p}=\mathbf{1}$. First assume that $\mathbf{x}_{A_1}\cdot \mathbf{p}=1$ and $\mathbf{x}_{A_2}\cdot \mathbf{p}=0.$ Then, by \eqref{eqn1}, $N\mathbf{p}=\overline{\mathbf{x}_{A_2}}$. By Lemma \ref{lemx} $A_2$ is solvable, which is a contradiction.

Second assume that $\mathbf{x}_{A_1}\cdot \mathbf{p}=1$ and $\mathbf{x}_{A_2}\cdot \mathbf{p}=1.$ Then, by \eqref{eqn1}, $N\mathbf{p}=\overline{\mathbf{x}_{A_1\cup A_2}}$.  By Lemma \ref{lemx}, $A_1 \cup A_2$ is solvable, i.e.; it is not $HO$. On the other hand, $A_1$ is $NO$ and $A_2$ is $HO$ imply $A_1 \cup A_2 $ is $HO$. Thus, we have a contradiction.

So either $\mathbf{x}_{A_1}\cdot \mathbf{p}=\mathbf{x}_{A_2}\cdot \mathbf{p}=0,$ which,  by \eqref{eqn1}, implies $N\mathbf{p}=\mathbf{1}$ with $\mathbf{x}_{A_2}\cdot \mathbf{p}=0$; or $\mathbf{x}_{A_1}\cdot \mathbf{p}=0$ and $\mathbf{x}_{A_2}\cdot \mathbf{p}=1,$ which,  by \eqref{eqn1}, implies $N\mathbf{p}=\overline{\mathbf{x}_{A_1}} $ with $\mathbf{x}_{A_2}\cdot \mathbf{p}=1$.

Conversely, first assume that $N\mathbf{p}=\mathbf{1}$ with $\mathbf{x}_{A_2}\cdot \mathbf{p}=0$. Since $A_1$ is $NO$ in $G$, we have $\mathbf{x}_{A_1}\cdot \mathbf{p}=0$. Hence,  by \eqref{eqn1}, $N^*\mathbf{p}=\mathbf{1}$.
Second assume that $N\mathbf{p}=\overline{\mathbf{x}_{A_1}} $ with $\mathbf{x}_{A_2}\cdot \mathbf{p}=1$. Then, by Lemma \ref{lem2}, $\mathbf{x}_{A_1}\cdot \mathbf{p}=0$. Thus, by \eqref{eqn1}, $N^*\mathbf{p}= \overline{\mathbf{x}_{A_1}}+\mathbf{x}_{A_1} = \mathbf{1}.$

Since $A_2$ is $HO$ in $G$, there are $2^{\nu(G)-1}$ patterns satisfying $N\mathbf{p}=\mathbf{1} \; \text{with}\; \mathbf{x}_{A_2}\cdot \mathbf{p}=0$, and also there are  $2^{\nu(G)-1}$ patterns satisfying $N\mathbf{p}=\overline{\mathbf{x}_{A_1}}\;\text{with} \; \mathbf{x}_{A_2}\cdot \mathbf{p}=1.$ So there are $2^{\nu(G)}$ patterns satisfying one of the two cases. On the other hand, there are $2^{\nu(G^*)}$ odd dominating patterns of $G^*$. Hence, $\nu(G^*)=\nu(G)$, which gives $\Delta\nu=0$. In each case,  $\mathbf{x}_{A_1}\cdot \mathbf{p}=0$, which makes $A_1$ a $NO$ set in $G^*$. On the other hand, $\mathbf{x}_{A_2}\cdot \mathbf{p}=0$ in the first case and $\mathbf{x}_{A_2}\cdot \mathbf{p}=1$ in the second one, which makes  $A_2$ a $HO$ set in $G^*$.
\end{proof}

\begin{prop}
\label{AH}
Let $A_1$ be an $AO$ and $A_2$ be a $HO$ set $G$. Then for any pattern $\mathbf{p}$, $N^*\mathbf{p}=\mathbf{1}\; \text{if and only if}\;  N\mathbf{p}=\overline{\mathbf{x}_{A_1}}\;\text{with} \; \mathbf{x}_{A_2}\cdot \mathbf{p}=1.$ Moreover, $A_1$ is a $NO$ and $A_2$ is an $AO$ set $G^*$, and $\Delta \nu =-1$.
\end{prop}
\begin{proof}
Let $N^*\mathbf{p}=\mathbf{1}$. First assume that $\mathbf{x}_{A_1}\cdot \mathbf{p}=\mathbf{x}_{A_2}\cdot \mathbf{p}=0.$ Then, by \eqref{eqn1}, $N\mathbf{p}=\mathbf{1}$. Then, $\mathbf{x}_{A_1}\cdot \mathbf{p}=0$ since $A_1$ is $AO$, which is a contradiction.

Second assume that $\mathbf{x}_{A_1}\cdot \mathbf{p}=1$ and $\mathbf{x}_{A_2}\cdot \mathbf{p}=0.$ Then, by \eqref{eqn1}, $N\mathbf{p}=\overline{\mathbf{x}_{A_2}}$. By Lemma \ref{lemx}, $A_2$ is solvable, which is a contradiction.

Third assume that $\mathbf{x}_{A_1}\cdot \mathbf{p}=1$ and $\mathbf{x}_{A_2}\cdot \mathbf{p}=1.$ Then, by \eqref{eqn1}, $N\mathbf{p}=\overline{\mathbf{x}_{A_1 \cup A_2}}$. By Lemma \ref{lemx}, $A_1 \cup A_2$ is solvable. On the other hand, $A_1$ is $AO$ and $A_2$ is $HO$ imply $A_1 \cup A_2 $ is $HO$, which means it is not solvable. Thus, we have a contradiction.

So the only possible case is  $\mathbf{x}_{A_1}\cdot \mathbf{p}=0$ and $\mathbf{x}_{A_2}\cdot \mathbf{p}=1,$ which by \eqref{eqn1}, gives $N\mathbf{p}=\overline{\mathbf{x}_{A_1}}$ with $\mathbf{x}_{A_2}\cdot \mathbf{p}=1$.

Conversely assume that  $N\mathbf{p}=\overline{\mathbf{x}_{A_1}}$ with $\mathbf{x}_{A_2}\cdot \mathbf{p}=1$. By Lemma \ref{lem2}, $\mathbf{x}_{A_1}\cdot \mathbf{p}=0$. Hence, by \eqref{eqn1}, $N^*\mathbf{p}=\overline{\mathbf{x}_{A_1}}+\mathbf{x}_{A_1}=\mathbf{1}$.

Since $A_2$ is $HO$ in $G$, there are $2^{\nu(G)-1}$ patterns satisfying $N\mathbf{p}=\overline{\mathbf{x}_{A_1}} \; \text{with}\; \mathbf{x}_{A_2}\cdot \mathbf{p}=1$, On the other hand, there are $2^{\nu(G^*)}$ odd dominating patterns of $G^*$. Hence, $\nu(G^*)=\nu(G)-1$, which gives $\Delta\nu=-1$. The above equivalence shows that when $N^*\mathbf{p}=\mathbf{1}$, $\mathbf{x}_{A_1}\cdot \mathbf{p}=0$ and $\mathbf{x}_{A_2}\cdot \mathbf{p}=1$, which makes $A_1$ a $NO$ and $A_2$  an $AO$ set $G^*$.
\end{proof}

\begin{lem}
\label{lem4}
Let $G$ be a graph, $\mathbf{c}$ be a solvable configuration in $G$ and $S$ be the set of all solving patterns for $\mathbf{c}$ with $|S|=n$. Let $A_1$ and $A_2$ be two disjoint HO sets in $G$. For $i\in\{1,2\}$, define the sets $O_i=\{\mathbf{p}\in S\; | \;   \mathbf{x}_{A_i}\cdot \mathbf{p}=0 \}$ and $I_i=\{\mathbf{p}\in S\; | \;   \mathbf{x}_{A_i}\cdot \mathbf{p}=1 \}$. Then, $|O_1\cap O_2|=|I_1\cap I_2|,$ and $A_1\cup A_2$ is $\mathbf{c}$-NO, HO, or $\mathbf{c}$-AO if and only if  $|O_1\cap O_2|=n/2, \;n/4,\;\text{or}\; 0,$ respectively.
\end{lem}

\begin{proof}

First note that, by Proposition \ref{prop2}, $|O_1|=|I_1|=|O_2|=|I_2|=n/2$ since $A_1$ and $A_2$ are $HO$. Moreover, $$I_1= S\cap I_1 =(O_2\cap I_1) \cup (I_2\cap I_1)$$
$$O_2= S\cap O_2 =(O_1\cap O_2) \cup (I_1\cap O_2),$$
which imply $n/2=|O_2\cap I_1|+|I_2\cap I_1|=|O_1\cap O_2|+|I_1\cap O_2|$. Hence, $|O_1\cap O_2|=|I_1\cap I_2|.$
 Let $O_{12}:=\{\mathbf{p}\in S\; | \;   \mathbf{x}_{A_1\cup A_2}\cdot \mathbf{p}=0 \}$.  Since  $\mathbf{x}_{A_1\cup A_2}\cdot \mathbf{p}=\mathbf{x}_{A_1}\cdot \mathbf{p}+\mathbf{x}_{A_2}\cdot \mathbf{p}$, we have $O_{12}= (O_1\cap O_2) \cup (I_1\cap I_2)$. Hence, $|O_{12}|=|O_1\cap O_2|+|I_1\cap I_2|= 2 |O_1\cap O_2|$. Now, the result follows by the fact that $A_1\cup A_2$ is $\mathbf{c}$-$NO$, $HO$, or $\mathbf{c}$-$AO$ if and only if  $|O_{12}|=n, \;n/2,\;\text{or}\; 0,$ respectively.

\end{proof}
\begin{prop}
\label{HHH}
Let   $A_1$, $A_2$, and $A_1 \cup A_2$ be $HO$ sets in $G$. Then for any pattern $\mathbf{p}$, $N^*\mathbf{p}=\mathbf{1}\; \text{if and only if}\;  N\mathbf{p}=\mathbf{1}$ with $\mathbf{x}_{A_1}\cdot \mathbf{p}=\mathbf{x}_{A_2}\cdot \mathbf{p}=0$. Moreover, $A_1$ and $A_2$ are $NO$ sets in $G^*$ and $\Delta\nu=-2$.
\end{prop}
\begin{proof}
Let $N^*\mathbf{p}=\mathbf{1}$. Assume that $\mathbf{x}_{A_1}\cdot \mathbf{p}= 0$ and $\mathbf{x}_{A_2}\cdot \mathbf{p}=1.$ Then, by \eqref{eqn1}, $N\mathbf{p}=\overline{\mathbf{x}_{A_1}}$.
By Lemma \ref{lemx} $A_1$ is solvable, which is a contradiction.

Assuming $\mathbf{x}_{A_1}\cdot \mathbf{p}= 1$ and $\mathbf{x}_{A_2}\cdot \mathbf{p}=0$ leads to a similar contradiction.

Assuming $\mathbf{x}_{A_1}\cdot \mathbf{p}= 1$ and $\mathbf{x}_{A_2}\cdot \mathbf{p}=1$ we obtain $N\mathbf{p}=\overline{\mathbf{x}_{A_1 \cup A_2}}$ by \eqref{eqn1}. Hence $A_1 \cup A_2$ is solvable, which is a contradiction.

So the only possible case is  $\mathbf{x}_{A_1}\cdot \mathbf{p}=0$ and $\mathbf{x}_{A_2}\cdot \mathbf{p}=0,$ which by \eqref{eqn1}, gives $N\mathbf{p}=\mathbf{1}$.

Conversely, if $\mathbf{p}$ satisfies the condition $N\mathbf{p}=\mathbf{1}$ with $\mathbf{x}_{A_1}\cdot \mathbf{p}=0$ and $\mathbf{x}_{A_2}\cdot \mathbf{p}=0$ then we get $N^*\mathbf{p}=\mathbf{1}$ by \eqref{eqn1}.

 Note that the above condition on $\mathbf{p}$ is equivalent to say that $\mathbf{p} \in O_1\cap O_2 $, where $O_1$ and $O_2$ are the sets defined as in Lemma \ref{lem4} with $\mathbf{c}=\mathbf{1}$. If we denote the number of odd dominating patterns of $G$ by $n$, then the number of patterns satisfying the condition is $n/4$ by Lemma \ref{lem4}. This is equivalent to say that $\Delta\nu=-2$ since the number of odd dominating patterns of $G$ and $G^*$ are $2^{\nu(G)}$ and $2^{\nu(G^*)}$, respectively.

  Lastly, since $\mathbf{x}_{A_1}\cdot \mathbf{p}=\mathbf{x}_{A_2}\cdot \mathbf{p}=0$ for all odd dominating patterns $\mathbf{p}$ of $G^*$, we conclude that $A_1$ and $A_2$ are $NO$ in $G^*$.

\end{proof}

\begin{prop}
\label{HHA}
Let   $A_1$, $A_2$ be $HO$, and $A_1 \cup A_2$ be $AO$ sets in $G$. Then for any pattern $\mathbf{p}$, $N^*\mathbf{p}=\mathbf{1}\; \text{if and only if}\;  N\mathbf{p}=\overline{\mathbf{x}_{A_1\cup A_2}}$ with $\mathbf{x}_{A_1}\cdot \mathbf{p}=\mathbf{x}_{A_2}\cdot \mathbf{p}=1$. Moreover, $A_1$ and $A_2$ are $AO$ sets in $G^*$ and $\Delta\nu=-1$.
\end{prop}
\begin{proof}
Let $N^*\mathbf{p}=\mathbf{1}$. Assuming $\mathbf{x}_{A_1}\cdot \mathbf{p}= 0$, $\mathbf{x}_{A_2}\cdot \mathbf{p}=1$ or $\mathbf{x}_{A_1}\cdot \mathbf{p}=1$, $\mathbf{x}_{A_2}\cdot \mathbf{p}=0$ leads us to the contradictions as in the proof of Proposition \ref{HHH}.

Assuming $\mathbf{x}_{A_1}\cdot \mathbf{p}=\mathbf{x}_{A_2}\cdot \mathbf{p}=0$, we obtain $N\mathbf{p}=\mathbf{1}$ by \eqref{eqn1}. This gives $\mathbf{x}_{A_1\cup A_2} \cdot \mathbf{p} =1$ since $A_1\cup A_2$ is $AO$ in $G$. However, by assumption $\mathbf{x}_{A_1\cup A_2} \cdot \mathbf{p} =\mathbf{x}_{A_1} \cdot \mathbf{p} +\mathbf{x}_{A_2} \cdot \mathbf{p} = 0$, which gives us a contradiction.

So the only possible case is  $\mathbf{x}_{A_1}\cdot \mathbf{p}=\mathbf{x}_{A_2}\cdot \mathbf{p}=1,$ which by \eqref{eqn1}, gives $N\mathbf{p}=\overline{\mathbf{x}_{A_1\cup A_2}}$.

Conversely, if $\mathbf{p}$ satisfies the condition $N\mathbf{p}=\overline{\mathbf{x}_{A_1\cup A_2}}$ with  $\mathbf{x}_{A_1}\cdot \mathbf{p}=\mathbf{x}_{A_2}\cdot \mathbf{p}=1$, then $N^*\mathbf{p}=\mathbf{1}$ by \eqref{eqn1}.

Note that the above condition on $\mathbf{p}$ is equivalent to say that $\mathbf{p} \in I_1\cap I_2 $, where $I_1$ and $I_2$ are the sets defined as in Lemma \ref{lem4} with $\mathbf{c}=\overline{\mathbf{x}_{A_1\cup A_2}}$. Note that $A_1\cup A_2$ is $\overline{\mathbf{x}_{A_1\cup A_2}}$-$NO$ by Lemma \ref{lem2}. Therefore, if we denote the number of patterns satisfying $N\mathbf{p}=\overline{\mathbf{x}_{A_1\cup A_2}}$  by $n$, then the number of patterns satisfying the condition is $n/2$ by Lemma \ref{lem4}. This is equivalent to say that $\Delta\nu=-1$ since the number of patterns satisfying $N\mathbf{p}=\overline{\mathbf{x}_{A_1\cup A_2}}$ is $2^{\nu(G)}$  and the number of odd dominating patterns of $G^*$ is $2^{\nu(G^*)}$.

Lastly, since $\mathbf{x}_{A_1}\cdot \mathbf{p}=\mathbf{x}_{A_2}\cdot \mathbf{p}=1$ for all odd dominating patterns $\mathbf{p}$ of $G^*$, we conclude that $A_1$ and $A_2$ are $AO$ in $G^*$.
\end{proof}

\begin{prop}
\label{HHN}
Let   $A_1$, $A_2$ be $HO$, and $A_1 \cup A_2$ be $NO$ sets in $G$. Then for any pattern $\mathbf{p}$, $N^*\mathbf{p}=\mathbf{1}\; \text{if and only if}\; N\mathbf{p}=\mathbf{1}$ with $\mathbf{x}_{A_1}\cdot \mathbf{p}=\mathbf{x}_{A_2}\cdot \mathbf{p}=0$ or $N\mathbf{p}=\overline{\mathbf{x}_{A_1\cup A_2}}$ with $\mathbf{x}_{A_1}\cdot \mathbf{p}=\mathbf{x}_{A_2}\cdot \mathbf{p}=1$. Moreover, $A_1$ and $A_2$ are $HO$ sets in $G^*$ and $\Delta\nu=0$.
\end{prop}
\begin{proof}
Let $N^*\mathbf{p}=\mathbf{1}$. Assuming $\mathbf{x}_{A_1}\cdot \mathbf{p}= 0$, $\mathbf{x}_{A_2}\cdot \mathbf{p}=1$ or $\mathbf{x}_{A_1}\cdot \mathbf{p}=1$, $\mathbf{x}_{A_2}\cdot \mathbf{p}=0$ leads us to the contradictions as in the proof of Proposition \ref{HHH}. Therefore, either $\mathbf{x}_{A_1}\cdot \mathbf{p}=\mathbf{x}_{A_2}\cdot \mathbf{p}=0$ or  $\mathbf{x}_{A_1}\cdot \mathbf{p}=\mathbf{x}_{A_2}\cdot \mathbf{p}=1$, which, by \eqref{eqn1}, gives $N\mathbf{p}=\mathbf{1}$ or $N\mathbf{p}=\overline{\mathbf{x}_{A_1\cup A_2}}$, respectively.

Conversely, if $\mathbf{p}$ satisfies the conditions $N\mathbf{p}=\mathbf{1}$ with $\mathbf{x}_{A_1}\cdot \mathbf{p}=\mathbf{x}_{A_2}\cdot \mathbf{p}=0$ or $N\mathbf{p}=\overline{\mathbf{x}_{A_1\cup A_2}}$ with  $\mathbf{x}_{A_1}\cdot \mathbf{p}=\mathbf{x}_{A_2}\cdot \mathbf{p}=1$, then in both cases $N^*\mathbf{p}=\mathbf{1}$ by \eqref{eqn1}.

 Note that the first condition on $\mathbf{p}$ is equivalent to say that $\mathbf{p} \in O_1\cap O_2 $, where $O_1$ and $O_2$ are the sets defined as in Lemma \ref{lem4} with $\mathbf{c}=\mathbf{1}$. $A_1 \cup A_2$ is a $NO$ set by assumption. So if we denote the number of odd dominating patterns of $G$ by $n$, then the number of patterns satisfying the first condition is $n/2$ by Lemma \ref{lem4}. The second condition on $\mathbf{p}$ is equivalent to say that $\mathbf{p} \in I_1\cap I_2 $, where $I_1$ and $I_2$ are the sets defined as in Lemma \ref{lem4} with $\mathbf{c}=\overline{\mathbf{x}_{A_1\cup A_2}}$. Note that $A_1\cup A_2$ is $\overline{\mathbf{x}_{A_1\cup A_2}}$-$NO$ by Lemma \ref{lem2}. Also, the number of the solving patterns for $\overline{\mathbf{x}_{A_1\cup A_2}}$ is equal to number of odd dominating patterns of $G$, which is $n$. Hence the number of patterns satisfying the second condition is $n/2$ by Lemma \ref{lem4}. So in total, there are $n$ patterns satisfying the first or second condition. Therefore, the number of odd dominating patterns of $G$, which is  $2^{\nu(G)}$ is equal to the number of odd dominating patterns of $G^*$, which is  $2^{\nu(G^*)}$. Hence $\Delta\nu=0$.

 We see that odd dominating patterns $\mathbf{p}$ of $G^*$ satisfy the first or second condition. In the first condition $\mathbf{x}_{A_1}\cdot \mathbf{p}=\mathbf{x}_{A_2}\cdot \mathbf{p}=0$ while in the second condition $\mathbf{x}_{A_1}\cdot \mathbf{p}=\mathbf{x}_{A_2}\cdot \mathbf{p}=1$ which makes $A_1$ and $A_2$  $HO$ sets in $G^*$.
\end{proof}

We summarize the results obtained from Proposition \ref{NN0} to \ref{HHN} in the following table.
$$$$
$$$$
$$$$
$$$$
$$$$
$$$$
$$$$
$$$$
\begin{table}[h]
\label{table1}
\caption{} 
\centering 
\begin{tabular}{|c|c|c|c|c|c|c|c|} 
\hline\hline 
\multirow{2}{*}{$A_1$ in $G$} & \multirow{2}{*}{$A_2$ in $G$} & \multirow{2}{*}{$A_1$ in $G^*$} & \multirow{2}{*}{$A_2$ in $G^*$} &  \multirow{2}{*}{$\Delta\nu$} & \multicolumn{2}{c|}{when}\\
\cline{6-7}
&  &  &  &  & $A_2$ in $G$ & \small $A_1 \cup A_2$ in $G$ \\
\hline 
\multirow{2}{*}{$NO$} & \multirow{2}{*}{$NO$}  & $NO$ & $NO$ & 0 & $\overline{\mathbf{x}_{A_1}}$-$NO$ & - \\ 
\cline{3-7}
       &  & $HO$ & $HO$ & 2 & $\overline{\mathbf{x}_{A_1}}$-$AO$ & - \\ 
\hline
\multirow{2}{*}{$NO$} & \multirow{2}{*}{$AO$}  & $AO$ & $HO$ & 1 & $\overline{\mathbf{x}_{A_1}}$-$NO$ & - \\ 
\cline{3-7}
       &  & $NO$ & $AO$ & 0 & $\overline{\mathbf{x}_{A_1}}$-$AO$ & - \\ 
\hline
\multirow{2}{*}{$AO$} & \multirow{2}{*}{$AO$}  & $AO$ & $AO$ & 0 & $\overline{\mathbf{x}_{A_1}}$-$NO$ & - \\ 
\cline{3-7}
       &  & $HO$ & $HO$ & 1 & $\overline{\mathbf{x}_{A_1}}$-$AO$ & - \\ 
\hline
\multirow{3}{*}{$HO$} & \multirow{3}{*}{$HO$}  & $NO$ & $NO$ & -2 & - & $HO$ \\ 
\cline{3-7}
       &  & $AO$ & $AO$ & -1 & - & $AO$ \\
\cline{3-7}
&  & $HO$ & $HO$ & 0 & - & $NO$ \\
\hline
$NO$ & $HO$ & $NO$ & $HO$ & $0$ & - & -     \\[0ex]
\hline
$AO$ & $HO$ & $NO$ & $AO$ & -1 & - & -     \\[0ex]
\hline

\end{tabular}
\label{table1}
\end{table}

\section{Main Results}
In this section we investigate the existence of edges whose addition to or removal from a graph increases or decreases the nullity of the graph under some specific conditions. The first natural question we ask is the following.
For a graph with positive nullity should there exist an edge whose removal decreases the nullity of the graph? As we proved in the following theorem the answer is affirmative.
\begin{thm}
\label{thmr}
Suppose $G$ is a graph with positive nullity. Then there exists an edge of $G$ whose removal decreases the nullity.
\end{thm}
\begin{proof}
Suppose for a contradiction that there does not exist such an edge. Since $\nu(G)>0$, there exists a nonzero null pattern. Hence there is a half activated vertex $u$ of $G$ (see Remark \ref{rem}). Let $\mathbf{s}$ be an odd dominating pattern of $G$ with $\mathbf{s}(u)=0$. For any half activated neighbor $v$ of $u$, $\mathbf{s}(v)=0$. Otherwise the set $\{u,v\}$ would be either $HO$ or $AO$, in which case the removal of the edge $uv$ decreases the nullity by Proposition \ref{HHH} and \ref{HHA}. On the other hand $u$ does not have any always activated neighbor. Otherwise, removal of the edge between $u$ and that activated vertex would decrease the nullity by Proposition \ref{AH}. Hence  $\mathbf{s}(w)=0$ for all neighbors $w$ of $u$. Then $\mathbf{x}_{N[u]}\cdot \mathbf{s}=0$, which contradicts with $\mathbf{s}$ being an odd dominating pattern.
\end{proof}
As a second question, we can ask the converse of the first one. For a graph $G$ which is not complete, does there always exist a pair of non adjacent vertices $u$ and $v$ such that the addition of the edge $uv$ to $G$ increases the nullity? The answer is an immediate no. Because, for example, any edge addition to the graph $K_2 \cup K_2$ only decreases the nullity. But what if the nullity of $G$ cannot decrease, i.e.; what if $G$ is always solvable? We can see that the answer is still no by considering the cycle $C_5$. Cycle $C_5$ is always solvable, on the other hand adding any edge to $C_5$ keeps the nullity zero. Note that $C_5$ satisfies some special properties such as having an odd order and being an \emph{even graph} i.e.; the degrees of all of its vertices are even. It turns out that if the nullity of an always solvable graph $G$ does not increase by any edge addition  then $G$ must satisfy these two properties as we prove in Theorem \ref{thma}. Before starting the proof we need the following definitions and lemma.
\begin{defn}
Let $G$ be a graph with vertices $u$ and $v$. If $u$ and $v$ are not adjacent then we denote by $G+uv$ the graph obtained from $G$ by joining $u$ and $v$ by an edge. If $u$ and $v$ are adjacent then we denote by $G-uv$ the graph obtained from $G$ by removing the edge between $u$ and $v$.
\end{defn}
\begin{defn}
We define the \emph{parity value} $pr(\mathbf{x})$ of a vector $\mathbf{x} \in \mathbb{Z}_2^n$ as $pr(\mathbf{x}):= \mathbf{1}\cdot\mathbf{x}$.
\end{defn}
Hence the parity value of a vector is $0$ if it has even number of nonzero coordinates and  $1$ if it has odd number of nonzero coordinates.

\begin{lem}
\label{lem6}
Let $G$ be an always solvable graph, $u$ be a vertex of $G$, $\mathbf{p}$ be the solving pattern for $\overline{\mathbf{x}_{\{u\}}}$, and $\mathbf{s}$ be the odd dominating pattern of $G$. Then,
\begin{align}
\label{Ns1}
\overline{\mathbf{x}_{N[u]}}\cdot \mathbf{p}= \left\{ \begin{array}{cc}
                pr(\mathbf{s}) &  \text{if}\;\;u\;\text{is never activated} \\
                1-pr(\mathbf{s}) &  \text{if}\;\;u\;\text{is always activated} \\
                \end{array} \nonumber \right\}.
\end{align}
\end{lem}
\begin{proof}
First we take an enumeration of vertices of $G$ such that $u$ corresponds to the first vertex. Now let $\mathbf{r}$ be the solving pattern for $\mathbf{x}_{\{u\}}$. Hence, $\mathbf{r}=N^{-1}\mathbf{x}_{\{u\}}$ which implies $\mathbf{r}$ is the first column vector of $N^{-1}.$ On the other hand, $\mathbf{s}=N^{-1}\mathbf{1}$, which implies $\mathbf{s}(u)$ is equal to the parity value of the first row vector of $N^{-1}.$ Since $N^{-1}$ is a symmetric matrix these two observations imply $pr(\mathbf{r})=\mathbf{s}(u)$. Furthermore, note that $\mathbf{p}=\mathbf{r}+\mathbf{s}$. Indeed, $\mathbf{p}=N^{-1}\overline{\mathbf{x}_{\{u\}}}=N^{-1}(\mathbf{x}_{\{u\}}+\mathbf{1})=N^{-1}\mathbf{x}_{\{u\}}+N^{-1}\mathbf{1}=\mathbf{r}+\mathbf{s}$. Hence $pr(\mathbf{p})=pr(\mathbf{r})+pr(\mathbf{s})= \mathbf{s}(u)+pr(\mathbf{s}).$ So $pr(\mathbf{p})=pr(\mathbf{s})$ if $u$ is never activated and $pr(\mathbf{p})=1+pr(\mathbf{s})=1-pr(\mathbf{s})$ if $u$ is always activated. On the other hand, $pr(\mathbf{p})= \mathbf{1}\cdot\mathbf{p} = \mathbf{x}_{N[u]}\cdot \mathbf{p} + \overline{\mathbf{x}_{N[u]}}\cdot \mathbf{p}= \overline{\mathbf{x}_{N[u]}}\cdot \mathbf{p}$ since $\mathbf{x}_{N[u]}\cdot \mathbf{p} =(N\mathbf{p})(u)=\overline{\mathbf{x}_{\{u\}}}(u)=0.$  Hence the result follows.
\end{proof}

\begin{thm}
\label{thma}
Suppose $G$ is an always solvable graph which is not an even graph with odd order. Then, there exists a non-adjacent pair of vertices $u$, $v$  in $G$ such that $G+uv$ is not always solvable.
\end{thm}
\begin{proof}
First of all note that an always solvable graph has no half activated vertices. So all vertices are either always activated or never activated. Now assume for a contradiction that for all non-adjacent pair of vertices $u$ and $v$,  $G+uv$ is always solvable. Since $G$ is not an even graph with odd order, either $G$ is not an even graph or $G$ is an even graph with even order.

Assume first that $G$ is not an even graph. Then $\mathbf{1}$ cannot be the odd dominating pattern. Hence there exists a never activated vertex $u$ of $G$. Let $\mathbf{p}$ be the solving pattern for $\overline{\mathbf{x}_{\{u\}}}$. Let $v$ be another vertex in $G$ which is not adjacent to $u$. Since $\nu(G+uv)=\nu(G)=0$, by Proposition \ref{NN1} and \ref{NA0}, $\mathbf{p}(v)= \mathbf{x}_{\{v\}}\cdot \mathbf{p}  =0$ if $v$ is never activated and $\mathbf{p}(v)=1$ if $v$ is always activated. Hence $\overline{\mathbf{x}_{N[u]}}\cdot \mathbf{p}$ is equal to the parity value of the always activated vertices which are not adjacent to $u$. Note that a never activated vertex has odd number of always activated neighbors in an always solvable graph. Hence the parity value of the always activated vertices which are not adjacent to $u$ is equal to $1-pr(\mathbf{s})$, where $\mathbf{s}$ is the odd dominating pattern of $G$. On the other hand, we have $\overline{\mathbf{x}_{N[u]}}\cdot \mathbf{p}=pr(\mathbf{s})$ by Lemma \ref{lem6} since $u$ is a never activated vertex. So we arrive at a contradiction.

Assume second that $G$ is an even graph with even order. Then $\mathbf{1}$ is the odd dominating pattern of $G$, in other words every vertex is always activated. Let $\tilde{u}$ be one of these vertices. Let $\tilde{v}$ be a vertex in $G$ which is not adjacent to $\tilde{u}$. Then, since $\tilde{v}$ is also always activated and $\nu(G+\tilde{u}\tilde{v})=\nu(G)=0$, by Proposition \ref{AA1}, $\mathbf{\tilde{p}}(\tilde{v})=0$, where $\mathbf{\tilde{p}}$ is the solving pattern for $\overline{\mathbf{x}_{\{\tilde{u}\}}}$. Hence  $\overline{\mathbf{x}_{N[\tilde{u}]}}\cdot \mathbf{\tilde{p}}=0$. However, by Lemma \ref{lem6}, $\overline{\mathbf{x}_{N[\tilde{u}]}}\cdot \mathbf{\tilde{p}}=1-pr(\mathbf{s})=1$ since the order of $G$ is even and every vertex is always activated. So we arrive at a contradiction for this case as well.
\end{proof}

Note that the converse of Theorem  \ref{thma} does not hold since the empty graph $\overline{K_3}$ is an even graph with odd order but every edge addition to it increases the nullity.

Theorem \ref{thma} answers the question of existence of an edge addition which increases the nullity of an always solvable graph. Similarly, we can ask whether there exists an edge removal which increases the nullity of an always solvable graph. The answer may be no if the graph is even. For example, $C_4$ is always solvable and any edge removal keeps the nullity zero. But if the graph is not even, then the answer is yes as the following theorem states.

\begin{thm}
\label{thmsss}
Suppose $G$ is an always solvable graph which is not an even graph. Then, there exists an edge $e$ of $G$ such that $G-e$ is not always solvable.
\end{thm}
\begin{proof}
Assume for a contradiction that there does not exist any edge whose removal increases the nullity. Since $G$ is not an even graph, $\mathbf{1}$ is not the odd dominating pattern. Hence there exists a never activated vertex $u$ of $G$. Let $\mathbf{p}$ be the solving pattern for $\overline{\mathbf{x}_{\{u\}}}$. Let $v$ be a neighbor of $u$ in $G$. By assumption,  $\nu(G-uv)=\nu(G)=0$. Hence, by Proposition \ref{NN1} and \ref{NA0}, $\mathbf{p}(v)= \mathbf{x}_{\{v\}}\cdot \mathbf{p}  =0$ if $v$ is never activated and $\mathbf{p}(v)=1$ if $v$ is always activated. On the other hand, $\mathbf{p}(u)=\mathbf{x}_{\{u\}}\cdot \mathbf{p}=0$ by Lemma \ref{lem2}. Moreover, $u$ must have odd number of always activated neighbors. Therefore, $\mathbf{x}_{N[u]}\cdot \mathbf{p}=1$. However, $\mathbf{x}_{N[u]}\cdot \mathbf{p}=(N\mathbf{p})(u)=\overline{\mathbf{x}_{\{u\}}}(u)=0$, which is a contradiction.
\end{proof}

\section{Further Results}

Theorem \ref{thmr} tells us that for a graph with positive nullity, there exists an edge whose removal decreases the nullity. However it does not tell us how much it decreases the nullity. From Table \ref{table1} we see that nullity can decrease by $1$ or by $2$. So another natural question is that for graphs with nullity greater than $1$ can we always find an edge whose removal decreases the nullity by $2$. The answer of this question is negative since $K_2 \cup K_2$ has nullity $2$ but removal of any edge decreases the nullity only by $1$. However, note that although we cannot find an edge of $K_2 \cup K_2$ whose removal decreases the nullity by $2$, there exist edges whose addition decreases the nullity by $2$. The following theorem proves that this is not a coincidence.

\begin{thm}
Suppose $G$ is a graph with $\nu(G)\geq 2$. Then either there exists an edge $e$ of $G$ such that $\nu(G-e)=\nu(G)-2$ or there are non-adjacent vertices $u$, $v$ of $G$ such that $\nu(G+uv)=\nu(G)-2$.
\end{thm}
\begin{proof}
Assume that the claim does not hold true. Then, for any pair of half activated vertices $u$, $v$, the set $\{u,v\}$ is not $HO$. Otherwise, removal or addition of the edge $uv$, depending on whether $u$ and $v$ are adjacent in $G$ or not, would decrease the nullity by $2$ by Proposition \ref{HHH}. This is equivalent to say that $\mathbf{x}_{\{u,v\}}$ is solvable. Hence $\mathbf{x}_{\{u,v\}}$ is orthogonal to every null pattern by (\emph{O 3}).

 Now, assuming $\nu(G)$ is nonzero, let $\boldsymbol{\ell}$ be a nonzero null pattern. Since $\boldsymbol{\ell}$ is nonzero there exists a half activated vertex $u$ such that $\boldsymbol{\ell}(u)=1$ (see Remark \ref{rem}). Let $v$ be any half activated vertex other than $u$. Then, $0=\mathbf{x}_{\{u,v\}}\cdot \boldsymbol{\ell}=\boldsymbol{\ell}(u)+\boldsymbol{\ell}(v)= 1+\boldsymbol{\ell}(v)$. Hence $\boldsymbol{\ell}(v)=1$. Moreover, $\boldsymbol{\ell}(w)=0$ for all $w$ which is either never or always activated. Otherwise $\{w\}$ would not be solvable by (\emph{O 3}). Hence $\boldsymbol{\ell}$ is the characteristic vector of the set of half activated vertices. Since $\boldsymbol{\ell}$ was an arbitrary nonzero null pattern, this implies $Ker(N)$ is one dimensional. So $\nu(G)=1$.
\end{proof}

Another result we have is the following.
\begin{thm}
Suppose $G$ is an always solvable graph where removal of every edge increases the nullity. Then, every always activated vertex of $G$ has even degree and every never activated vertex of $G$ has odd degree.
\end{thm}
\begin{proof}
Since $G$ is always solvable, all vertices are either always activated or never activated. Let $u$ be an always activated vertex of $G$. For any never activated neighbor $v$ of $u$, since the removal of the edge $uv$ increases the nullity, by Proposition \ref{NA1} and Lemma \ref{lem3}, $\mathbf{p}(v)= \mathbf{x}_{\{v\}}\cdot \mathbf{p}= 1$, where $\mathbf{p}$ is the solving pattern for $\overline{\mathbf{x}_{\{u\}}}$. Similarly, for any always activated neighbor $w$ of $u$, $\mathbf{p}(w)=1$ by Proposition \ref{AA0}. Moreover, $\mathbf{p}(u)=\mathbf{x}_{\{u\}}\cdot \mathbf{p}=0$ by Lemma \ref{lem2}. Consequently, $\mathbf{x}_{N[u]} \cdot \mathbf{p}=0$ if and only if the degree of $u$ is even. On the other hand, $\mathbf{x}_{N[u]} \cdot \mathbf{p}= (N\mathbf{p})(u)=\overline{\mathbf{x}_{\{u\}}}(u)=0.$ Hence degree of $u$ is even.

Let $\tilde{u}$ be a never activated vertex of $G$. For any never activated neighbor $\tilde{v}$ of $\tilde{u}$, $\mathbf{\tilde{p}}(\tilde{v})=1$ by Proposition \ref{NN0}, where $\mathbf{\tilde{p}}$ is the vector satisfying $N\mathbf{\tilde{p}}=\overline{\mathbf{x}_{\{\tilde{u}\}}}$. On the other hand, for any always activated neighbor $\tilde{w}$ of $\tilde{u}$, $\mathbf{\tilde{p}}(\tilde{w})=0$ by Proposition \ref{NA1}. $\mathbf{\tilde{p}}(\tilde{u})=\mathbf{x}_{\{\tilde{u}\}}\cdot \mathbf{\tilde{p}}=0$ by Lemma \ref{lem2}. Consequently, the number of never activated neighbors  of $\tilde{u}$ is even since $\mathbf{x}_{N[\tilde{u}]} \cdot \mathbf{\tilde{p}}=0$. On the other hand, since there is no half activated vertex, a never activated vertex has odd number of always activated neighbors. So the degree of $\tilde{u}$, which is the sum of never and always activated neighbors of $\tilde{u}$, is odd.
\end{proof}

We finish the set of our results by giving a characterization of always solvable graphs. First we need to name some types of edge additions.

\begin{defn}
Let $G$ be a graph with a non-adjacent pair of vertices $u$ and $v$. Let $\mathbf{p}$ be a solving pattern for $\overline{\mathbf{x}_{\{u\}}}$ in the case $\{u\}$ is solvable. Then adding the edge $uv$ to $G$ is called a

\emph{Type-1} edge addition if $u$ is never activated, $v$ is always activated, and $\mathbf{p}(v)=1$,

\emph{Type-2} edge addition if $u$ and $v$ are always activated, and $\mathbf{p}(v)=0$,

\emph{Type-3} edge addition if $u$ is never activated, $v$ is always activated, and $\mathbf{p}(v)=0$,

\emph{Type-4} edge addition if $u$ and $v$ are always activated, and $\mathbf{p}(v)=1$,

\emph{Type-5} edge addition if $u$ is always activated, $v$ is half activated,

\emph{Type-6} edge addition if $u$ and $v$ are half activated, and $\{u,v\}$ is half odd activated.
\end{defn}
\begin{thm}
\label{thmc}
Suppose $G$ is a graph with $n$ edges where $n\neq 0$. Then, $G$ is always solvable if and only if $G$ is obtained from an always solvable graph with $n-1$ edges by a Type-1 or Type-2 edge addition or it is obtained from an always solvable graph with $n-2$ edges by a Type-3 or a Type-4 edge addition followed by a Type-5 or Type-6 edge addition.
\end{thm}
\begin{proof}
If $G$ is always solvable then it has no half activated vertex. On the other hand, $G$ must have an always activated vertex $v$ with nonzero degree. Otherwise $G$ would have a never activated vertex $w$ with no always activated neighbors since $G$ is nonempty. But then $\mathbf{x}_{N[w]} \cdot \mathbf{s}=0$ where $\mathbf{s}$ is the odd dominating pattern of $G$, which is a contradiction. Therefore, by Proposition \ref{NA0}-\ref{AA1}, either (\emph{Case 1}) there exists an edge incident to $u$ whose removal keeps the nullity $0$, or (\emph{Case 2})  removal of every edge incident to $u$ increases the nullity by $1$.

First assume that \emph{Case 1} holds. Then there exists a vertex $u$ adjacent to $v$ such that $G':=G-uv$ is always solvable. By Proposition \ref{NA1} and  \ref{AA0}, types of activations of $u$ and $v$ in $G'$ are equal to the types of activations in $G$. So $u$ and $v$ is either a never activated, always activated or always activated, always activated pair in $G'$. Since the nullity of $G=G'+uv$ is the same as the nullity of $G'$ which is $0$, this implies $G$ is obtained from $G'$ by a \emph{Type-1} or \emph{Type-2} edge addition by Proposition \ref{NA0}-\ref{AA1}.

Second assume that \emph{Case 2} holds. Then $G'=G-uv$ has nullity $1$ for all $u$ adjacent to $v$ in $G$. Equivalently, adding the edge $uv$ to $G'$ decreases the nullity by $1$. Considering all possible scenarios stated in Proposition \ref{NN0}-\ref{HHN} (or put together in Table 1), we see that adding the edge $uv$ to $G'$ can decrease the nullity of $G'$ by $1$ only under two cases which corresponds \emph{Type-5} and \emph{Type-6} edge additions. On the other hand, since the nullity of $G'$ is $1$, there exists an edge $e$ of $G'$ whose removal decreases the nullity to $0$ by Theorem \ref{thmr}. Equivalently, adding the edge $e$ to the always solvable graph $G'':=G'-e$ increases the nullity by $1$. Again, considering all possible scenarios stated in Proposition \ref{NN0}-\ref{HHN}, we see that this is only possible under two cases which corresponds to \emph{Type-3} or \emph{Type-4} edge addition.

Sufficiency part of the proof is trivial by definitions of the types of the edge additions and by Proposition \ref{NA0}, \ref{NA1}, \ref{AA0}, \ref{AA1}, \ref{AH}, and \ref{HHA}
\end{proof}

\textbf{Declaration of Competing Interests} The author declares that he has no
known competing financial interests or personal relationships that could have appeared to influence the work reported in this paper. \medskip

\bibliographystyle{plain}

\end{document}